\documentclass[preprint,12pt]{elsarticle}
\usepackage{amsthm,amsfonts,amssymb,amscd,amsmath,enumerate,verbatim,calc,graphicx,geometry}
\usepackage[all]{xy}
\newtheorem{theorem}{Theorem}[section]
\newtheorem{lemma}[theorem]{Lemma}

\newtheorem{corollary}[theorem]{Corollary}
\theoremstyle{definition}
\theoremstyle{definitions}
\newtheorem{definition}[theorem]{Definition}

\newtheorem{example}[theorem]{Example}
\theoremstyle{notations}

\theoremstyle{remarks}

\newcommand{\st}{\stackrel}

\journal{ }

\begin{document}

\begin{frontmatter}



\title{On Hawaiian Groups of Some Topological Spaces}


\author[]{Ameneh~Babaee}
\ead{am.babaee88@gmail.com}
\author[]{Behrooz~Mashayekhy\corref{cor1}}
\ead{bmashf@um.ac.ir}
\author[]{Hanieh~Mirebrahimi}
\ead{h$_{-}$mirebrahimi@um.ac.ir}
\address{Department of Pure Mathematics, Center of Excellence in Analysis on Algebraic Structures, Ferdowsi University of
Mashhad,\\
P.O.Box 1159-91775, Mashhad, Iran.}
\cortext[cor1]{Corresponding author}
\begin{abstract}
The paper is devoted to study the structure of Hawaiian groups of some topological spaces. We present some behaviors of
Hawaiian groups with respect to product spaces,
weak join spaces, cone spaces, covering spaces and locally trivial bundles.
In particular, we determine the structure of the $n$-dimensional Hawaiian group of the $m$-dimensional Hawaiian earring
space, for all $1\leq m\leq n$.
\end{abstract}


\end{frontmatter}


\section{Introduction and Motivation}
In 2000, K. Eda and K. Kawamura \cite{eda} defined the $n$-dimensional Hawaiian earring, $n=1,2,\ldots$, as the following
subspace of the $(n+1)$-dimensional Euclidean space ${\Bbb{R}}^{(n+1)}$
\begin{center}
$\Bbb{H}^n = \{(r_0,r_1,...,r_n)\in {\Bbb{R}}^{(n+1)}\ |\ (r_0-1/k)^2+{\sum}_{i=1}^n r_i^2 = (1/k)^2, k\in {\Bbb{N}}\}.$
\end{center}
Here $\theta = (0,0,...,0)$ is regarded as the base point of $\Bbb{H}^n$, and $S_k^n$ shows the $n$-sphere in $\Bbb{H}^n$
with radius 1/k.


In 2006, U.H. Karimov and D. Repov\u{s} \cite{karh} defined a new notion, the  $n$-Hawaiian group of
a pointed space $(X,x_0)$ to be the set of all pointed  homotopy classes $[f]$, where
$f:({\Bbb{H}^n},\theta)\rightarrow(X,x_0)$ is continuous, with a group operation which comes naturally from the operation
of $n$th homotopy group denoted by ${\mathcal{H}}_n(X,x_0)$
which we call it the $n$-Hawaiian group of $(X,x_0)$. This group is homotopy invariant in the category of all pointed topological
spaces. One can see that ${\mathcal{H}}_n:hTop_{\ast} \rightarrow Groups $ is a covariant functor from the pointed homotopy category,
$hTop_{\ast}$, to the category of all groups, $Groups$, for $n\geq 1$. If
$\beta:(X,x_0)\rightarrow (Y,y_0)$ is a continuous map between pointed spaces, then
$\mathcal{H}_n(\beta)={\beta}_{\ast}:{\mathcal{H}}_n (X,x_0)\rightarrow {\mathcal{H}}_n (Y,y_0)$ defined by
${\beta}_{\ast}([f])= [\beta \circ f]$ is a homomorphism.

They also mentioned \cite{karh} some advantages of Hawaiian group functor rather than other famous functors such as  homotopy,
homology and cohomology functors. There exists a contractible space $C(\Bbb{H}^1)$, the cone over $\Bbb{H}^1$, with nontrivial
$1$-Hawaiian group, but trivial homotopy, homology and cohomology groups. In fact in \cite{karh}, it is showed that
${\mathcal{H}}_1(C({\Bbb{H}}^1),\theta)$ is uncountable. Also, this functor can help us to get some local properties
of spaces. In fact, if $X$ has a countable system of neighborhood at $x_0$, then countability of the $n$-Hawaiian group ${\mathcal{H}}_n (X,x_0)$ implies $n$-locally simply connectedness of $X$ at $x_0$.
(see \cite[Theorem 2]{karh}).

There is a relation between the Hawaiian group and the homotopy groups of a pointed space $(X,x_0)$ as follows.
\begin{theorem} (\cite[Theorem 1]{karh}). If the space $X$ is $n$-locally simply connected at the point $x_0$ and satisfies the first
countability axiom, then
$$\hspace{3cm}{\varphi}:{\mathcal{H}}_n(X,x_0) \rightarrow  {\prod}_{i\in\Bbb{N}}^w{\pi}_n(X,x_0) \hspace{3cm}(I) $$
defined by $\varphi([f])=([f{\mid}_{S_1^n}], [f{\mid}_{S_2^n}],... )$ is an isomorphism, where
$\prod_{i\in\Bbb{N}}^w{\pi}_n(X,x_0)$ is the weak direct product of countable copies of ${\pi}_n(X,x_0)$.
\end{theorem}

We use the following concepts frequently in the paper.
\begin{definition} A space $X$ is called:\\
$(i)$  $n$-locally simply connected at the point $x_0$ if for every neighborhood $U\subset X$ of $x_0$ there is a neighborhood $V\subset U$ of $x_0$ such that the homomorphism $i_*:\pi_n(V,x_0)\rightarrow \pi_n(U,x_0)$ induced by inclusion is zero.\\
$(ii)$  $n$-semilocally simply connected at the point $x_0$ if there is a neighborhood $U$ of $x_0$ such that the homomorphism $i_*:\pi_n(U,x_0)\rightarrow \pi_n(X,x_0)$ induced by inclusion is zero.\\
$(iii)$ locally strongly contractible at $x_0$ if for every neighborhood $U\subset X$ of $x_0$ there exists a neighborhood $V\subset U$ of $x_0$ such that the inclusion map $V\hookrightarrow U$ is null-homotopic to the point $x_0$ \\
$(iv)$ semilocally strongly contractible at $x_0$ if there exists a neighborhood $U$ of $x_0$ such that the inclusion
map $U\hookrightarrow X$ is null-homotopic to the point $x_0$ (see \cite{eda}).
\end{definition}

The paper is organized as follows. In Section 2, we establish some more properties of Hawaiian groups. First we show that $n$-Hawaiian groups are abelian for all $n\geq 2$. We also show that the map $(I)$ is an isomorphism for semilocally strongly contractible spaces.
Second, we compute the $n$-Hawaiian group of a weak join of a countable family of $(n-1)$-connected, locally strongly
contractible and first countable pointed spaces (see \cite{mm} for the definition of the weak join).
As a consequence, we can compute the $m$-Hawaiian group of an $n$-Hawaiian earring for all $1\leq m\leq n$.
Moreover, we show that all Hawaiian group functors preserve direct products.
Third, we concentrate on Hawaiian groups of the cone of spaces. As a main result, we show that the $n$-Hawaiian group of the cone of $(X,x_0)$,
$C(X)$, at the point $(x_0,t)$ except the vertex is the quotient group of ${\mathcal{H}}_n(X,x_0)$ by ${\prod}_{i\in\Bbb{N}}^w{\pi}_n(X,x_0)$.
Finally, we give an exact sequence of Hawaiian groups for a first countable locally trivial bundle which gives more
information about Hawaiian groups of covering spaces or ${\Bbb{R}}^n$-bundles.


 Karimov and  Repov\u{s} \cite{karo} generalized the Hawaiian earring to the infinite dimension, to be the weak join of  all
 finite dimensional Hawaiian earrings and denoted by $\Bbb{H}^{\infty}$.
Then they followed it by the infinite dimensional Hawaiian group similar to finite dimension to be the set of all pointed
 homotopy classes $[f]$, where
$f:({\Bbb{H}}^{\infty},\theta)\rightarrow(X,x_0)$ is continuous, endowed with group operation, which comes from the
finite dimensional Hawaiian groups denoted by ${\mathcal{H}}_{\infty}(X,x_0)$. One can verify that ${\mathcal{H}}_{\infty}:hTop_\ast \rightarrow Groups $ is a
covariant functor with induced homomorphism similar to finite dimension cases.
Also they construct a Peano continuum with trivial homotopy, homology (singular, \u{C}ech and Borel\_Moore), cohomology
(singular and \u{C}ech) and finite dimensional Hawaiian groups, which is not contractible and has nontrivial
 infinite dimensional Hawaiian group \cite{karo}.


In Section 3, we study the structure of the infinite dimensional Hawaiian group. We extend most of properties of finite
dimensional Hawaiian groups obtained in Section 2 to the infinite case.

\section{Finite Dimensional Hawaiian Groups}
We say that a pointed space $(X,x_0)$ has a local property if $X$ has the property at point $x_0$.
\begin{definition}
Let $\{f_i:(X,x_0)\rightarrow (Y,y_0)|i\in I\}$ be a family of continuous maps. We say that $\{f_i\}_{i\in I}$ is
null-convergent if for each open set $U$ containing $y_0$, we have $Im(f_i)\subseteq U$ for all $i\in I$ except a finite
number.
\end{definition}
The following lemma will be used in several results.
\begin{lemma}\label{joinmap}
With the previous notations and assumptions, let $(X,x_0)$ be a pointed topological space, then the following statements hold.\\
$(i)$ Let $\{f_k:(S_k^n,a)\rightarrow (X,x_0)\}$ be a  sequence of continuous maps, then
$f:(\Bbb{H}^n,\theta)\rightarrow (X,x_0)$ defined by $f{\mid}_{S_k^n}=f_k$ is continuous
if and only if $\{f_k\}$ is null-convergent.\\
$(ii)$ Let $\{f_k,f'_k:(S_k^n,a)\rightarrow (X,x_0)\}_{k\in\Bbb{N}}$ be two sequences of continuous maps with
$f_k\simeq f'_k\ rel\ \{a\}$ by null-convergent sequence of homotopies $\{H_k:S_k^n\times I\rightarrow X\}_{k\in\Bbb{N}}$
relative to the point $a$. Then $H:f\simeq f'\ rel\ \{\theta\}$, where  $H{\mid}_{S_k^n\times I}=H_k$, $f{\mid}_{S_k^n}=
f_k$ and $f'{\mid}_{S_k^n}=f'_k$.\\
$(iii)$ Let $f_1,\ f_2,\ ...,\ f_m$ and $f'_1,\ f'_2,\ ...,\ f'_m$ be two  finite sequences of continuous  maps with
$f_k,f'_k:(S_k^n,a)\rightarrow (X,x_0)$ and $f_k\simeq f'_k$ for $k=1,2,\ldots,m$ and let
 $g,g':({\tilde{\bigvee}}_{k\geq m+1}S_k^n,a)\rightarrow (X,x_0)$ be two continuous homotopic maps relative to
 $\{\theta\}$, where ${\tilde{\bigvee}}_{k\geq m+1}S_k^n$ is the weak join of the family of $n$-spheres $\{S_k^n|k\geq m\}$ . Then $f,f':(\Bbb{H}^n,\theta)\rightarrow (X,x_0)$ defined by $f{\mid}_{S_k^n}=f_k$,
  $f'{\mid}_{S_k^n}=f'_k$, for $k=1,2,\ldots,m$ and $f{\mid}_{{\tilde{\bigvee}}_{k\geq m+1}S_k^n}=g$,
  $f'{\mid}_{{\tilde{\bigvee}}_{k\geq m+1}S_k^n}=g'$ are continuous and homotopic relative to $\{\theta\}$.
\end{lemma}
\begin{proof}
$(i)$ Let  $f:(\Bbb{H}^n,\theta)\rightarrow (X,x_0)$ be continuous, then for each neighborhood $U$ of $x_0$, $f^{-1}(U)$
contains all the $n$-sphere's of the Hawaiian earring except a finite number, that is, there exists $K\in \Bbb{N}$ such
that $Im(f{\mid}_{S_k^n})\subseteq U$ for all $k\geq K$. Hence if $k\geq K$, then $Im(f_k)\subseteq U$ and so $\{f_k\}$ is
a null-convergent sequence.

Now, let $\{f_k\}$ be a null-convergent sequence, then if $U$ is an open set containing $x_0$, there exists $K\in \Bbb{N}$
such that for all $k\geq K$, $Im(f_k)\subseteq U$. Since $f{\mid}_{S_k^n}=f_k$ , for $k\geq K$ we have
$Im(f{\mid}_{S_k^n})\subseteq U$ or equivalently $S_k^n\subseteq f^{-1}(U)$. Using the topology of weak join on the
Hawaiian earring and continuity of the $f_k$'s, $f$ is continuous.\\
$(ii)$ Similar to the proof of part $(i)$ one can prove that $H$ is continuous and verify that $H:f\simeq f'\ rel\ \{\theta\}$.\\
$(iii)$ Since $g$ is continuous, then by $(i)$ $\{g{\mid}_{S_{m+1}^n},\ g{\mid}_{S_{m+2}^n},\ ...\}$ is a null-convergent
sequence and so is
$\{f_1,\ ,f_2,\ ...,\ f_m,\ g{\mid}_{S_{m+1}^n},\ g{\mid}_{S_{m+2}^n},\ ...\}$. Using again $(i)$ implies that $f$ is
continuous.
Now, let $F_k:f_k\simeq f'_k$, $G:g\simeq g'$ and define $H{\mid}_{S_k^n\times I}=F_k$, for $k=1,\ldots,m$ and
$H{\mid}_{({\widetilde{\bigvee}}_{k\geq m+1}S_k^n)\times I}=G$.
Using the topology of the Hawaiian earring, $H$ is continuous and hence $H:f\simeq f'\ rel\ \{\theta\}$.
\end{proof}


\begin{theorem}\label{abel}
Let $(X,x_0)$ be a pointed space, then ${\mathcal{H}}_n(X,x_0)$ is an abelian group, for all $n\geq 2$.
\end{theorem}
\begin{proof}
If $[f],[g]\in {\mathcal{H}}_n(X,x_0)$, then  $[f]\ast [g]=[f\ast g]$ and  $(f\ast g){\mid}_{S_k^n}=f{\mid}_{S_k^n}\ast g{\mid}_{S_k^n}$. Since
$n\geq 2$, we have a homotopy map $H_k:f{\mid}_{S_k^n}\ast g{\mid}_{S_k^n}\simeq g{\mid}_{S_k^n}\ast f{\mid}_{S_k^n}\ rel\ \{\theta\}$ with $Im(H_k)=Im(f{\mid}_{S_k^n})\cup Im(g{\mid}_{S_k^n})$ (see \cite[page 340]{hatch} for further details). Since $f$, $g$ are continuous, by Lemma \ref{joinmap} $(i)$, for each open set $U$ containing $x_0$, there exist $K_f,K_g\in\Bbb{N}$ such that for $k\geq K_f$, $Im(f{\mid}_{S_k^n})\subseteq U$ and for $k\geq K_g$, $Im(g{\mid}_{S_k^n})\subseteq U$. Thus for $k\geq max\{K_f,K_g\}$, $Im(f{\mid}_{S_k^n})\cup Im(g{\mid}_{S_k^n})\subseteq U$ and so $Im(H_k)\subseteq U$ which  implies  that $\{H_k\}_{k\in\Bbb{N}}$ is null-convergent. Use Lemma \ref{joinmap} $(ii)$ to prove that $f\ast g\simeq g\ast f\ rel \ \{\theta\}$.
\end{proof}


\begin{lemma}\label{weakdir}
If $(X,x_0)$ is a pointed space, then ${\prod}_{i\in \Bbb{N}}^w{\pi}_n(X,x_0)$ can be embedded in ${\mathcal{H}}_n(X,x_0)$.
\end{lemma}
\begin{proof}
Define $\beta:{\prod}_{i\in \Bbb{N}}^w{\pi}_n(X,x_0)\longrightarrow {\mathcal{H}}_n(X,x_0)$ by  $$\beta([f_1],[f_2],...,[f_m],e,e,...)=[f],$$ in which $f{\mid}_{S_k^n}=f_k$ for $k\leq m$, and  $f{\mid}_{S_k^n}=c_{x_0}$ for $k>m$, where $c_{x_0}$ is the constant loop at $x_0$.
Since the sequence $\{f{\mid}_{S_k^n}\}_{k\in\Bbb{N}}$ is constant except a finite number, it is null-convergent  and  Lemma \ref{joinmap} implies that $f$ is continuous. Let $\beta([f_1],[f_2],...,[f_m],e,e,...)=e$, then the corresponding map $f$ is null-homotopic relative to $\{\theta\}$ and so $f{\mid}_{S_k^n}=f_k\simeq c_{x_0}\ \ rel\{\theta\} $.
Hence $\beta$ is injective. The operation of  ${\mathcal{H}}_n(X,x_0)$ implies that $\beta$ is a homomorphism and hence it is a monomorphism.
\end{proof}


\begin{theorem}\label{semistrnog}
Let $(X,x_0)$ be a semilocally strongly contractible pointed space, then
$${\mathcal{H}}_n(X,x_0)\cong{\prod}_{i\in\Bbb{N}}^w{\pi}_n(X,x_0).$$
\end{theorem}
\begin{proof}
Consider the homomorphism $\varphi:{\mathcal{H}}_n(X,x_0)\rightarrow {\prod}{\pi}_n(X,x_0)$ defined in $(I)$. Since $X$ is semilocally strongly contractible at $x_0$, there exists an open set $U$ containing $x_0$ with null-homotopic inclusion map $i:U\hookrightarrow X$ relative to the point $x_0$.
For each continuous map $f:({\Bbb{H}}^n,\theta)\longrightarrow (X,x_0)$, there exists  $K\in\Bbb{N}$ such that if $k\geq K$, then $Im(f{\mid}_{S_k^n})\subseteq U$ and so $Im(f{\mid}_{{\tilde{\bigvee}}_{k\geq K}S_k^n})\subseteq U$. Therefore $i\circ f{\mid}_{{\tilde{\bigvee}}_{k\geq K}S_k^n}$ is null-homotopic in $X$ relative to the point $\theta$. Since $i$ is the inclusion map and $Im(f{\mid}_{{\tilde{\bigvee}}_{k\geq K}S_k^n})\subseteq U$, $f{\mid}_{{\tilde{\bigvee}}_{k\geq K}S_k^n}$ is null-homotopic in $X$ relative to the point $\theta$. Hence if $k\geq K$, then $f{\mid}_{S_k^n}$ is null-homotopic relative to the point $\theta$ and so $\varphi$ maps $f$ to an element of ${\prod}_{i\in\Bbb{N}}^w{\pi}_n(X,x_0)$. Thus $Im(\varphi)\subseteq{\prod}_{i\in\Bbb{N}}^w{\pi}_n(X,x_0)$ and easily seen that the equality holds.

For injectivity, let $[f]\in{\mathcal{H}}_n(X,x_0)$ with $\varphi([f])=(e,e,...)$, using the above technique, there exists $K\in \Bbb{N}$ such that  $f{\mid}_{{\tilde{\bigvee}}_{k\geq K}S_k^n}$ is null-homotopic in $X$ relative to the point $\theta$. Since $\varphi([f])=(e,e,...)$, $f{\mid}_{S_k^n}\simeq c_{x_0}\ \ rel\{\theta\}$, for all $k<K$. Now, Lemma \ref{joinmap} $(iii)$ implies that $f$ is null-homotopic relative to the point $\theta$.
\end{proof}


Note that by the above theorem we can compute the Hawaiian group of some more spaces than Theorem 1.1. As an example, the join space $X={\bigvee}_{i\in \Bbb{N}}S^n$
is not first countable but using Theorem \ref{semistrnog} we have $${\mathcal{H}}_n(X,x_0)\cong{\prod}_{i\in\Bbb{N}}^w{\pi}_n(X,x_0).$$


\begin{definition}
Let $(X,x_0)$ be a pointed space and $n\geq 1$. We define $L_n(X,x_0)$ to be a subset  of ${\prod}_{i\in\Bbb{N}}\pi_n(X,x_0)$ consisting of all sequences of homotopy classes $\{[f_i]\}$,
where $\{f_i\}$ is null-convergent.
\end{definition}


\begin{theorem}
Let $(X,x_0)$ be a pointed space and $\varphi:{\mathcal{H}}_n(X,x_0)\rightarrow {\prod}{\pi}_n(X,x_0)$ be the homomorphism (I), then $Im(\varphi)=L_n(X,x_0)$. In particular, $L_n(X,x_0)$ is a subgroup of ${\prod}_{i\in\Bbb{N}}{\pi}_n(X,x_0)$.
 \end{theorem}
\begin{proof}
 Let $f:({\Bbb{H}^n},\theta)\rightarrow(X,x_0)$ be a map and $U$ be an open set in $X$ containing $x_0$. Since $f$ is continuous by Lemma \ref{joinmap} $(i)$,  $\{f{\mid}_{S_k^n}\}_{k\in\Bbb{N}}$ is null-convergent. Therefore $Im\varphi\subseteq L_n(X,x_0)$. Conversely,
 let $\{f_k:(S_k^n,\theta)\rightarrow(X,x_0)\}$ be a null-convergent sequence. Put $f{\mid}_{S_k^n}=f_k$, then by Lemma \ref{joinmap} $(i)$ $f$ is continuous and we have $\varphi([f])=([f_1],[f_2],...)$. Thus $\varphi:{\mathcal{H}}_n(X,x_0)\longrightarrow L_n(X,x_0)$ is an epimorphism and the result holds.
\end{proof}


\begin{lemma}
Let $(X,x_0)$ be $n$-semilocally simply connected, then
$$L_n(X,x_0)={\prod}_{i\in\Bbb{N}}^w{\pi}_n(X,x_0).$$
In particular, if $n\geq2$ and  $\varphi$ is the homomorphism $(I)$, then
$${\mathcal{H}}_n(X,x_0)\cong{\prod}_{i\in\Bbb{N}}^w{\pi}_n(X,x_0)\oplus Ker(\varphi).$$
\end{lemma}
\begin{proof}
Since $X$ is $n$-semilocally simply connected at $x_0$, there exists a neighborhood $U$ of $x_0$ such that the homomorphism $i_{\ast}:{\pi}_n(U,x_0)\rightarrow {\pi}_n(X,x_0)$  induced by the inclusion map $i:U\rightarrow X$ is trivial. If $\{[f_k]\}\in L_n(X,x_0)$, then there exists $K$ such that $Im(f_k)\subseteq U$, for all $k\geq K$. Since  $i_\ast:{\pi}_n(U,x_0)\rightarrow {\pi}_n(X,x_0)$ is the trivial homomorphism,  $i\circ f_k$ and so $f_k$ is null-homotopic in $X$, for all  $k\geq K$. Hence $L_n(X,x_0)\subseteq{\prod}_{i\in\Bbb{N}}^w{\pi}_n(X,x_0)$,
one can easily see the reverse inclusion.

 Consider the following exact sequence
$$0\rightarrow Ker(\varphi)\rightarrow {\mathcal{H}}_n(X,x_0)\st{\varphi}{\rightarrow} {\prod}_{i\in\Bbb{N}}^w{\pi}_n(X,x_0)\rightarrow 0.$$
If $n\geq2$, then by Theorem \ref{abel}, ${\mathcal{H}}_n(X,x_0)$ is abelian. It is routine to check that $\beta\circ\varphi=id$, where
$\beta:{\prod}_{i\in \Bbb{N}}^w{\pi}_n(X,x_0)\longrightarrow {\mathcal{H}}_n(X,x_0)$ is the monomorphism defined in the proof of Lemma \ref{weakdir}.
 Hence we have ${\mathcal{H}}_n(X,x_0)\cong{\prod}_{i\in\Bbb{N}}^w{\pi}_n(X,x_0){\oplus} Ker(\varphi)$.
 \end{proof}


\begin{theorem}\label{main}
Let ${\{(X_i,x_i)\}}_{i\in{\Bbb{N}}}$ be a family of locally strongly contractible, first countable pointed spaces. If $X={\widetilde{\bigvee}}_{i\in{\Bbb{N}}}(X_i,x_i)$
is the weak join of the family ${\{(X_i,x_i)\}}_{i\in{\Bbb{N}}}$ and $x_{\ast}$ is the common point, then $${\mathcal{H}}_n(X,x_\ast) \cong L_n(X,x_{\ast}).$$
\end{theorem}
\begin{proof}
Let  $\varphi:{\mathcal{H}}_n(X,x_\ast)\rightarrow L_n(X,x_\ast)$ be the natural epimorphism and $[f]$ be an element of ${\mathcal{H}}_n(X,x_\ast)$ so that $\varphi([f])=(e,e,...)$.  Then $f{\mid}_{S^n_k}\simeq c_{x_\ast}\ rel \{\theta\}$ with a homotopy mapping $F_k:{S_k^n}\times I\rightarrow X$, for $k\in\Bbb{N}$. Since $X_i$ is locally strongly contractible and first countable at $x_i$, for $i\in\Bbb{N}$, there exists a nested local basis $\{V_j^i\}_{j\in\Bbb{N}}$ at $x_i$ such that the inclusion map $J_j^i:V_j^i\rightarrow V_{j-1}^i$ is null-homotopic to $x_i$. Let $\{U_m\}_{m\in\Bbb{N}}$ be a  local basis at $x_\ast$ obtained by $U_m=({\bigvee}_{i<m}V_m^i)\bigvee({\widetilde{\bigvee}}_{i\geq m}X_i)$. For each $m\in \Bbb{N}$, there is $K_m$ such that if $k\geq K_m$, then $Im(f{\mid}_{S^n_k})\subseteq U_{m}$. Let $R_m:X\rightarrow {\widetilde{\bigvee}}_{i\geq m}X_i$ be the natural contraction. We show that $f{\mid}_{S^n_k}$ is  homotopic to the map $f_k:=R_m\circ f{\mid}_{S^n_k}\ rel\ \{\theta\}$ in $U_m$. To prove this, let $J_m^i:V_m^i\rightarrow V^i _{m-1}$ be the inclusion map which is null-homotopic with homotopy mapping $H_m^i$. Consider the inclusion map $J_m:U_m\longrightarrow U_{m-1}$ and define $H_m:U_m\times I\rightarrow U_{m-1}$ by joining of all the mapping $H_m^i$'s, for $i<m$, and identity map for the others. So $H_m:J_m\simeq R_m\ rel\ \{x_\ast\}$  and hence  $H_m(f{\mid}_{S^n_k},-):J_m\circ f{\mid}_{S^n_k}\simeq R_m\circ f{\mid}_{S^n_k}\ rel\ \{\theta\}$. We define $\tilde{f}{\mid}_{S^n_k}:=f{\mid}_{S^n_k}$ and $G_k:{S_k^n}\times I\rightarrow X$ by $G_k=F_k$  when $k<K_1$. For $K_m\leq k<K_{m+1}$, we define  $\tilde{f}{\mid}_{S^n_k}=f_k$ and $ G_k:{S_k^n}\times I\rightarrow {\widetilde{\bigvee}}_{i\geq m}X_i$ by $G_k=R_m oF_k$. Finally, we define $G:{\Bbb{H}}^n\times I\rightarrow X$ by $G {\mid}_{S^n_k\times I}=G_k$. Using Lemma \ref{joinmap} $(ii)$, $G$ is continuous and $G:\tilde{f}\simeq c \ rel\ \{\theta\}$. Since $f\simeq\tilde{f} \ rel\ \{\theta\}$, the result holds.
\end{proof}


 Eda and Kawamura \cite[Theorem 1.1]{eda} proved the following result.

\emph{Let} $n\geq2$ \emph{and} $X_i$\emph{ be an }$(n-1)$-\emph{connected, Tikhonov space which is semilocally strongly contractible at }$x_i$, \emph{for each} $i\in I$.\emph{ If} $x_\ast$\emph{ is the common point, then} $${\pi}_n({\widetilde{\bigvee}}_{i\in I}(X_i,x_i),x_\ast)\cong {\widetilde{\prod}}_{i\in I}{\pi}_n(X_i,x_i),\ \ \ \ \ \ \ (II)$$
\emph{where }${\widetilde{\prod}}_{i\in I}{\pi}_n(X_i,x_i)$\emph{ is a subgroup of }${\prod}_{i\in I}{\pi}_n(X_i,x_i)$ \emph{consisting of all }$f$'s \emph{such that the set }$\{i\in I: f(i)\neq 0\}$\emph{ is countable.}

We can improve the above result as follows when the set of indexes is countable.


\begin{theorem}\label{mainH}
Let $n\geq2$ and $\{(X_i,x_i)\}_{i\in\Bbb{N}}$ be a family of $(n-1)$-connected Tikhonov spaces which are locally strongly contractible and first countable pointed spaces. If $X={\widetilde{\bigvee}}_{i\in{\Bbb{N}}}(X_i,x_i)$ and $x_{\ast}$ is the common point, then
$${\mathcal{H}}_n(X,x_\ast) \cong{\prod}_{i\in\Bbb{N}}{\prod}_{k\in\Bbb{N}}^w{\pi}_n(X_i,x_i).$$
\end{theorem}
\begin{proof}
Using Theorem \ref{main} we have the monomorphism $\varphi:{\mathcal{H}}_n(X,x_\ast) \rightarrow{\prod}_{k\in\Bbb{N}}{\pi}_n(X,x_\ast)$ and by $(II)$ there exists an isomorphism $h:{\pi}_n(X,x_\ast)\rightarrow{\prod}_{i\in\Bbb{N}}{\pi}_n(X_i,x_i)$ defined by $h([g])=\{[r_i \circ g]\}_{i\in\Bbb{N}}$ when $r_i:X\rightarrow X_i$ is the natural contraction. Also, there is a natural isomorphism  ${\prod}_{k\in\Bbb{N}}{\prod}_{i\in\Bbb{N}}{\pi}_n(X_i,x_i)\cong{\prod}_{i\in\Bbb{N}}{\prod}_{k\in\Bbb{N}}{\pi}_n(X_i,x_i)$. Now, by the composition of these homomorphisms, we obtain a monomorphism $\psi:{\mathcal{H}}_n(X,x_\ast)\rightarrow{\prod}_{i\in\Bbb{N}}{\prod}_{k\in\Bbb{N}}{\pi}_n(X_i,x_i)$ with the rule $\psi([f])=\{\{[r_i \circ f{\mid}_{S^n_k}]\}_{k\in\Bbb{N}}\}_{i\in\Bbb{N}}$. We show that $Im(\psi)\subseteq{\prod}_{i\in\Bbb{N}}{\prod}_{k\in\Bbb{N}}^w{\pi}_n(X_i,x_i)$. Let $[f]\in {\mathcal{H}}_n(X,x_\ast)$ and let $\{U_m\}$ be the local basis defined in the proof of Theorem \ref{main}. By Lemma \ref{joinmap} there exists an increasing  sequence $\{K_i\}_{i\in\Bbb{N}}$ such that if $k\geq K_i$, then $Im(f{\mid}_{S^n_k})\subseteq U_{i+1}$. So $Im(r_{i}\circ f{\mid}_{S^n_k})\subseteq V_{i+1}^i$ and hence $r_{i}\circ f{\mid}_{S^n_k}$ is null-homotopic relative to $\{\theta\}$ in $X_i$. Thus for each $i\in\Bbb{N}$, all terms of the sequence $\{[r_i \circ f{\mid}_{S^n_k}]\}_{k\in\Bbb{N}}$  are trivial except a finite number. Hence $\{[r_i \circ f{\mid}_{S^n_k}]\}_{k\in\Bbb{N}}\in{\prod}_{k\in\Bbb{N}}^w{\pi}_n(X_i,x_i)$ and $\{\{[r_i \circ f{\mid}_{S^n_k}]\}_{k\in\Bbb{N}}\}_{i\in\Bbb{N}}\in {\prod}_{i\in\Bbb{N}}{\prod}_{k\in\Bbb{N}}^w{\pi}_n(X_i,x_i)$.

Now, let $\{\{[f_k^i]\}_{k\in\Bbb{N}}\}_{i\in\Bbb{N}}\in{\prod}_{i\in\Bbb{N}}{\prod}_{k\in\Bbb{N}}^w{\pi}_n(X_i,x_i)$ such that for each $i\in\Bbb{N}$, the sequence $f_k^i$ is the constant map except a finite number. Consider $\{[f_k^i]\}_{i\in\Bbb{N}}\}_{k\in\Bbb{N}}\in{\prod}_{k\in\Bbb{N}}{\prod}_{i\in\Bbb{N}}{\pi}_n(X_i,x_i)$, use  the convergent sequence $\{1/2^i\}_{i\in\Bbb{N}}$ to define  a product on infinite terms of $n$-loops. Let $f_k=f_k^1*f_k^2*...$, then $r_i\circ f_k=c_{x_1}*...*c_{x_{i-1}}*f_k^i*c_{x_{i+1}}*...\cong f_k^i$, so $h([f_k])=\{[f_k^i]\}_{k\in\Bbb{N}}$. We claim that $\{f_k\}_{k\in\Bbb{N}}$ is null-convergent. Given an open set $U$ containing $x_0$, there exists $U_m$ defined in the proof of Theorem \ref{main} such that $U_m\subseteq U$. For each $i\in\Bbb{N}$, there exists $K_i$ such that if $k\geq K_i$, then $f_k^i=c_{x_i}$. Now for each $m\in\Bbb{N}$, define $K=max\{K_1,k_2,...,K_{m-1}\}$. For $i\leq m-1$, if $k\geq K$, then $f_k^i=c_{x_i}$ and hence $Im(f_k^i)\subseteq U_m$. For $i\geq m$, we have $X_i\subseteq U_m$ and so $Im(f_k^i)\subseteq U_m$. Since $Im(f_k)={\bigcup}_{i\in\Bbb{N}}Im(f_k^i)$, if $k\geq K$, then $Im(f_k)\subseteq U_m$.
\end{proof}


In \cite[Corollary 1.2]{eda} Eda and Kawamura computed homotopy and  homology groups of Hawaiian earring spaces as follows.

\emph{For the $m$-dimensional Hawaiian earring ${\Bbb{H}}^m$, $m\geq 2$, $\pi_n({\Bbb{H}}^m,\theta)$ is trivial for each $n$, $1\leq n\leq m-1 $ and ${\pi}_m({\Bbb{H}}^m,\theta)\cong {\Bbb{Z}}^\omega$, where ${\Bbb{Z}}^\omega$ means the direct product of countably many the infinite cyclic group, ${\Bbb{Z}}$.}

The following corollary is a consequence of Theorem \ref{mainH} and the above fact.
\begin{corollary}\label{2.12}
Let $n\geq 2$, then ${\mathcal{H}}_n(\Bbb{H}^n, \theta) \cong {\prod}_{i\in\Bbb{N}}{\prod}_{k\in\Bbb{N}}^w\Bbb{Z}$ and ${\mathcal{H}}_m(\Bbb{H}^n, \theta)$ is trivial,
when $1\leq m\leq n-1$.
\end{corollary}

Note that if $\theta'\in \Bbb{H}^n$ and $\theta'\neq \theta$, then by \cite[Theorem 1.1]{eda}
${\mathcal{H}}_n(\Bbb{H}^n, \theta') \cong {\prod}^{w}_{k\in\Bbb{N}}{\prod}_{i\in\Bbb{N}}\Bbb{Z}$, for $n\geq 2$. Answering a question of Fuchs' problem 76, it is known that ${\prod}_{i\in\Bbb{N}}{\prod}_{k\in\Bbb{N}}^w\Bbb{Z}\not\cong {\prod}^{w}_{k\in\Bbb{N}}{\prod}_{i\in\Bbb{N}}\Bbb{Z}$ (see \cite{Z}).
Hence by Corollary \ref{2.12} we have
${\mathcal{H}}_n(\Bbb{H}^n, \theta') \not\cong {\mathcal{H}}_n(\Bbb{H}^n, \theta)$,
for $n\geq 2$.

\begin{theorem}\label{pr}
For any family of spaces $\{X_i\}_{i\in I}$ and $n\geq 1$, the following isomorphism holds.
$${\mathcal{H}}_n({\prod}_{i\in I}X_i,x_\ast)\cong {\prod}_{i\in I} {\mathcal{H}}_n(X_i,x_i),$$
where $x_\ast =\{x_i\}_{i\in I}\in {\prod}_{i\in I}X_i$ .
\end{theorem}
\begin{proof}
Consider projection maps $p_j:{\prod}_{i\in I}X_i\longrightarrow X_j$ and induced homomorphisms ${\mathcal{H}}_n(p_j):{\mathcal{H}}_n({\prod}_{i\in I}X_i,x_\ast)\longrightarrow{\mathcal{H}}_n(X_j,x_j)$, and define
$$\psi=\{{\mathcal{H}}_n(p_i)\}_{i\in I}:{\mathcal{H}}_n({\prod}_{i\in I}X_i,x_\ast)\longrightarrow {\prod}_{i\in I} {\mathcal{H}}_n(X_i,x_i)$$
by $\psi([f])=\{{\mathcal{H}}_n(p_i)([f])\}_{i\in I}$.
It is easy to see that this homomorphism has an inverse ${\psi}^{-1}:\{[f_i]\}_{i\in I}\longmapsto [f]$, where f is the unique continuous map which comes from the universal property of the product.
\end{proof}


Note that if ${\mathcal{H}}_n(X,x_0)$ is trivial, then so is $\pi_n(X,x_0)$, but the converse is not true. The cone over the Hawaiian earring \cite[Remark 1]{karh} is a counterexample.


\begin{theorem}\label{cone}
Let $(X,x_0)$ be a pointed space and $C(X)$ be the cone over $X$. Suppose that  $\hat{x}_t= (x_0,t)$ is a point of $C(X)$
except the vertex ($t\neq 1$). Then $${\mathcal{H}}_n(C(X),\hat{x}_t)\cong\frac{{\mathcal{H}}_n(X,x_0)}{{\prod}_{i\in \Bbb{N}}^w{\pi}_n(X,x_0)}.$$
\end{theorem}
\begin{proof}
Let $i_t:X\rightarrow C(X)$ be the natural inclusion by $i_t(x)=(x,t)$. There exists a homomorphism
$\mu={\mathcal{H}}_n(i_t):{\mathcal{H}}_n(X,x_0)\rightarrow {\mathcal{H}}_n(C(X),\hat{x}_t)$ by $\mu([f])=[i_t\circ f]$, for each  $[f]\in {\mathcal{H}}_n(X,x_0)$. Since $X\times [0,1)$ is an open set in $C(X)$, by the topology of $\Bbb{H}^n$, for any $[g] \in  {\mathcal{H}}_n(C(X),{\hat{x}}_t)$, there exists $K$ in $\Bbb{N}$ such that $Im(g{\mid}_{S_k^n})\subseteq X\times [0,1)$, for $k\geq K$. We define $\overline{g}:\Bbb{H}^n\rightarrow C(X)$ by $\overline{g}{\mid}_{S_k^n}=c_{x_t}$ for $k<K$ and $\overline{g}{\mid}_{S_k^n}= {g}{\mid}_{S_k^n}$ otherwise. We show that $[\overline{g}]=[g]$. Since ${\pi}_n(C(X),\hat{x}_t)$ is trivial, there exist homotopy mappings $A_k:{g}{\mid}_{S_k^n}\cong c_{\hat{x}_t}\ rel\ \{\theta\}$, for each $k\in\Bbb{N}$. We define $A:\Bbb{H}^n\times I\rightarrow C(X)$ by $A{\mid}_{S_K^n}\times I=A_k$, for $k<K$ and $A{\mid}_{S_K^n\times I}(r,s)= g(r)$, for $k\geq K$ which is continuous by Lemma \ref{joinmap} and makes $g$ and $\overline{g}$  homotopic relative to $\{\theta\}$. Since $Im(\overline{g})\subseteq X\times [0,1)$, one can consider $[\overline{g}]\in{\mathcal{H}}_n(X\times [0,1),{\hat{x}}_t)$. Theorem \ref{pr} states that ${\mathcal{H}}_n(p_1):{\mathcal{H}}_n(X\times [0,1),{\hat{x}}_t)  \cong{\mathcal{H}}_n(X,x_0) $ is an isomorphism with an inverse which is induced by the injection  $j_t:X\longrightarrow X\times [0,1)$ by the rule $j_t(x)=(x,t)$. Thus there exists $[f]\in {\mathcal{H}}_n(X,x_0)$  such that ${\mathcal{H}}_n(j_t)([f])=[\overline{g}]$. Since $j_t(x)=i_t(x)$ for each $x\in X$, $j\circ j_t\circ f=i_t\circ f$, where $j:X\times [0,1)\rightarrow C(X)$ is the inclusion map. Therefore we have
${\mathcal{H}}_n(i_t)([f])=[i_tof]=[joj_tof]={\mathcal{H}}_n(j)([j_tof])={\mathcal{H}}_n(j)([\overline{g}])=[\overline{g}]=[g]$.
Hence $\mu([f])=[g]$ and  $\mu$ is an epimorphism.

Moreover, ${\prod}^w_{i\in \Bbb{N}}{\pi}_n(X,x_0)\subseteq ker\mu$. To prove this, let $[f]\in{\mathcal{H}}_n(X,x_0)$ and $\beta([f])=([f_1],[f_2],...,[f_k],e,e,...)\in{\prod}^w_{i\in \Bbb{N}}{\pi}_n(X,x_0)$. Since ${\pi}_n(C(X),\hat{x}_t)$ is trivial, the homomorphism ${i_t}_\ast:{\pi}_n(X,x_0)\rightarrow {\pi}_n(C(X),\hat{x}_t)$ is trivial and so $[i_tof_i]$ is the identity element of ${\pi}_n(C(X),\hat{x}_t)$, for $i\leq K$.  Hence $i_t\circ f_1,i_t\circ f_2,...,i_t\circ f_k,c_{\hat{x}_t},...$ is a sequence of null-homotopic maps relative to $\{\theta\}$ which is constant except a finite number and by Lemma \ref{joinmap} $(ii)$ we have $i_t\circ f$ is null-homotopic relative to $\{\theta\}$. Hence $\mu([f]={\mathcal{H}}_n(i_t)[f]$ is the identity element of ${\mathcal{H}}_n(C(X),\hat{x}_t)$ and thus $[f]\in ker\mu$.

  Now, let $\mu([f])=e$, then $[i_tof]=e$. If $i_t\circ f$ is null-homotopic in $C(X)$ with given homotopy $H:i_tof\simeq c_{\hat{x}_t}\ rel\ \{\theta\}$, then since $V=X\times [0,1)$ is an open set in $C(X)$ containing $\hat{x}_t$ and $H$ is continuous, $H^{-1}(V)$ is an open set that contains $(\theta, s)$, for every $s\in I$. Therefore, there exists $U_s\times J_s\subseteq  H^{-1}(V)$ for each $(\theta, s)$ such that $s\in J_s$ and $\theta\in U_s$. $\{J_s\}$ has a finite subcover $\{J_{s_l}| \ 1\leq l\leq m\}$ for $I$. Let  $U={\bigcap}_{l=1}^mU_{s_l} $, since  $U\times J_{s_l}\subseteq U_{s_l}\times J_{s_l}\subseteq H^{-1}(V)$ and $U\times I={\bigcup}_{l=1}^mU_{s_l}$, $U\times I\subseteq H^{-1}(V)$. Using the form of open sets in  the space $\Bbb{H}^n$ at  the point $\theta$, there exists $K$ in $\Bbb{N}$ such that for $k\geq K$, $S_k^n\subseteq U$ and then ${S_k^n\times I}\subseteq H^{-1}V$ or equivalently $Im(H{\mid}_{S_k^n\times I})\subseteq V$. We have $Im(i_t\circ f{\mid}_{S_k^n})\subseteq V=X\times [0,1)$, for $k\geq K$. If $f{\mid}_{S_k^n}$ is not null-homotopic in $X$, then $i_t\circ f{\mid}_{S_k^n}$ is not null-homotopic in $X\times [0,1)$. But $ H|_{S_k^n\times I}:i_t\circ f{\mid}_{S_k^n}\simeq c\ rel\ \{\theta\}$ in $X\times [0,1)$, thus for $k\geq K$, $f{\mid}_{S_k^n}$ is  null-homotopic relative to $\{\theta\}$. Hence $ker(\mu)\subseteq {\prod}_{i\in \Bbb{N}}^w{\pi}_n(X,x_0)$ and then  $ker(\mu)={\prod}_{i\in \Bbb{N}}^w{\pi}_n(X,x_0)$.
\end{proof}


The following corollary is a consequence of Theorem \ref{cone} and \cite[Theorem 1]{karh}.
\begin{corollary}\label{conecor}
Let $n\geq 1$, $(X,x_0)$ be a pointed space, $C(X)$ be the cone over $X$ and $\hat{x}_t=(x_0,t)$, where $t\neq 1$. Then  $\Bbb{H}_n(C(X),\hat{x}_t)$ is trivial if one of the following conditions holds:\\
$(i)$ $(X,x_0)$ is $n$-locally simply connected and first countable.\\
$(ii)$ $(X,x_0)$ is semilocally strongly contractible.
\end{corollary}


\begin{example}\label{S}
Let $S={\widetilde{\bigvee}}_{n\in\Bbb{N}}S^n$ and $a$ be the common point, then by Corollary \ref{conecor}, $\Bbb{H}_n(CS,a)$ is trivial, for each natural number $n$.
\end{example}


In \cite[Theorem 1]{karh} it is proved that for a first countable pointed space $(X,x_0)$ and $n\geq 1$, $\varphi:{\mathcal{H}}_n(X,x_0) \rightarrow {\prod}^w{\pi}_n(X,x_0)$
is an isomorphism if $X$ is $n$-locally simply connected at $x_0$. The following result gives us more information.
\begin{corollary}
Let $(X,x_0)$ be a first countable pointed space  and $n\geq 1$, then the following statements are equivalent.\\
$(i)$  $X$ is $n$-locally simply connected at $x_0$.\\
$(ii)$ $\varphi:{\mathcal{H}}_n(X,x_0) \rightarrow {\prod}^w{\pi}_n(X,x_0)$  is an isomorphism.\\
$(iii)$ ${\mathcal{H}}_n(C(X),\hat{x}_t)$ is trivial,
where $C(X)$ is the cone over $X$, $t\neq1$ and $\hat{x}_t=(x_0,t)$.
\end{corollary}
\begin{proof}
$(i)\Rightarrow (ii)$. If $X$ is $n$-locally simply connected at $x_0$, then by \cite[Theorem 1]{karh} $\varphi:{\mathcal{H}}_n(X,x_0) \rightarrow {\prod}^w{\pi}_n(X,x_0)$ is an isomorphism.\\
$(ii)\Rightarrow (iii)$. Let $\varphi:{\mathcal{H}}_n(X,x_0) \rightarrow {\prod}^w{\pi}_n(X,x_0)$ be an isomorphism, then by Theorem \ref{cone} the $n$-Hawaiian group of the cone over the space $X$ is trivial.\\
$(iii)\Rightarrow (i)$. Let  ${\mathcal{H}}_n(C(X),\hat{x}_t)$ be trivial, then  by \cite[Theorem 2]{karh} $C(X)$ is $n$-locally simply connected at $\hat{x}_t=(x_0,t)$ ($t\neq 1$) and we show that so is $(X,x_0)$. By contrary, suppose that there exists a neighborhood $U$ of $x_0$ such that for each open set $V$ containing $x_0$ the inclusion map $V\hookrightarrow U$ does not induce trivial homomorphism and so is $V\times J\hookrightarrow U\times J$, for every subinterval $J$ of the unit interval $I$. This contradicts to the $n$-locally simply connectedness of $C(X)$ at $\hat{x}_t=(x_0,t)$. Therefore $(X,x_0)$ is $n$\_locally simply connected.
\end{proof}


\begin{corollary}
Let $(X,x_0)$ be a first countable pointed space, $t\neq1$ and $\hat{x}_t=(x_0,t)$, then
 ${\mathcal{H}}_n(C(X),\hat{x}_t)$ is trivial or is uncountable.
\end{corollary}
\begin{proof}
Let  ${\mathcal{H}}_n(C(X),\hat{x}_t)$ be a countable set, then $C(X)$ is $n$-locally simply connected at $\hat{x}_t$ by \cite[Theorem 2]{karh}. So $X$ is $n$-locally simply connected at $x_0$. Hence $\varphi$ is an isomorphism and then ${\mathcal{H}}_n(C(X),\hat{x}_t)$ is trivial.
\end{proof}


We need the following lemma to study the behavior of Hawaiian groups on locally trivial bundles.
\begin{lemma}\label{Homotopylift}
Let $p: E\rightarrow B$ be a locally trivial bundle, $y_0\in B$, $x_0\in p^{-1}(y_0)=F$ and  $(E,x_0)$ be first countable. Let  $\tilde{\alpha}:(\Bbb{H}^n,\theta)\rightarrow(E,x_0)$  lift ${\alpha}:(\Bbb{H}^n,\theta)\longrightarrow(B,y_0)$ and $F$ be a pointed homotopy with ${F}(-,0)={\alpha}$, then there exists a pointed  homotopy  $\tilde{F}$ that lifts  $F$ and so $\tilde{F}(-,0)=\tilde{\alpha}$.
\end{lemma}
\begin{proof}
Let  $\{U_i\}$ be a countable basis at $x_0$ in $E$. Since the projection map is open, $\{V_i=p(U_i)\}$ is a countable collection of open sets which is a local basis for $B$ at $y_0$. Similar to the proof of Theorem \ref{cone}, one can see that there exists an increasing sequence $K_1,K_2,K_3,...$ such that for $k\geq K_i$, $Im(F{\mid}_{S_k^n\times I})\subseteq V_i$. Homotopy lifting property of locally trivial bundles which holds on pointed homotopies implies that for each $k<K_1$, we can  define ${\tilde{F}}_k:S_k^n\times I\longrightarrow E$ to be the lifting  of $F{\mid}_{S_k^n\times I}$ such that ${\tilde{F}}_k(-, 0)=\tilde{\alpha}{\mid}_{S_k^n}$. The map $p{\mid}_{U_i}:U_i\longrightarrow V_i$ is again a locally trivial bundle and if $K_i\leq k<K_{i+1}$, then we have $Im(F{\mid}_{S_k^n\times I})\subseteq V_i$. So for each pointed homotopy $F{\mid}_{S_k^n\times I}$ there exists a lifting ${\tilde{F}}_k$ such that ${\tilde{F}}_k(-, 0)=\tilde{\alpha}{\mid}_{S_k^n}$ and $Im({\tilde{F}}_k)\subseteq U_i$. We define $\tilde{F}{\mid}_{S_k^n\times I}={\tilde{F}}_k$ which is continuous by Lemma \ref{joinmap} and it is a lifting for $F$ such that $\tilde{F}(-,0)=\alpha$. Since all $\tilde{F}{\mid}_{S_k^n\times I}$ may be chosen pointed  homotopies with $\tilde{F}{\mid}_{S_k^n\times I}(-,0)=\tilde{\alpha}{\mid}_{S_k^n\times I}$, then so is $\tilde{F}$.
\end{proof}


It is known that for each locally trivial bundle there exists an exact sequence of homotopy groups as follows
  $$...\rightarrow  {\pi}_{n}(F,y_0)\rightarrow {\pi}_n(E,y_0)\rightarrow{\pi}_n(B,x_0)\rightarrow {\pi}_{n-1}(F,y_0)\rightarrow$$ $$ \rightarrow {\pi}_{n-1}(E,y_0) \rightarrow ... \rightarrow {\pi}_{1}(F,y_0)\rightarrow {\pi}_{1}(E,y_0)\rightarrow {\pi}_{1}(B,y_0).$$
In the following theorem we gives a similar exact sequence of Hawaiian groups with some conditions.
\begin{theorem}
Let $p: E\rightarrow B$ be a locally trivial bundle and $E$ be first countable at $x_0\in p^{-1}(y_0)=F$, where $y_0\in B$. Then there exists an exact sequence of Hawaiian groups as follows
$$...\rightarrow  {\mathcal{H}}_{n}(F,y_0)\rightarrow {\mathcal{H}}_n(E,y_0)\rightarrow{\mathcal{H}}_n(B,x_0)\rightarrow {\mathcal{H}}_{n-1}(F,y_0)  \rightarrow $$ $$\rightarrow {\mathcal{H}}_{n-1}(E,y_0) \rightarrow...\rightarrow {\mathcal{H}}_{1}(F,y_0)\rightarrow {\mathcal{H}}_{1}(E,y_0)\rightarrow {\mathcal{H}}_{1}(B,y_0).$$
\end{theorem}
\begin{proof}
The inclusion  map $j:F\rightarrow E$ induces the homomorphism ${\mathcal{H}}_{n}(F,y_0)\rightarrow {\mathcal{H}}_n(E,y_0)$ and the projection $p:E\rightarrow B$ induces the homomorphism ${\mathcal{H}}_n(E,y_0)\rightarrow{\mathcal{H}}_n(B,x_0)$. Let $n\geq 2$, then we define  $\partial: {\mathcal{H}}_n(B,x_0)\rightarrow {\mathcal{H}}_{n-1}(F,y_0)$ and we prove that the sequence is exact.  Let $\alpha:(\Bbb{H}^n,\theta)\rightarrow(B,x_0)$ be an arbitrary continuous map, there exists an increasing sequence of natural numbers $K_1,K_2,K_3,...$ such that for each $k\geq K_i$, $Im (\alpha{\mid}_{S_k^n})\subseteq V_{i}$. Consider ${\alpha}_k=\alpha{\mid}_{S_k^n}$ and $S^n_k=\Sigma S^{n-1}_k$. Applying the Homotopy Lifting Theorem for pointed homotopies and $S^{n-1}_k$, we can deduce that there exists a lifting for ${\alpha}_k$. Similar to the proof of Theorem \ref{Homotopylift}, we construct a lifting $\tilde{\alpha}$ for $\alpha$ such that if $K_i\leq k<K_{i+1}$, then $\tilde{\alpha}{\mid}_{S_k^{n-1}}(-,1):S_k^{n-1}\rightarrow F\cap U_i$. Put $\bar{\alpha}:\Bbb{H}^{(n-1)}\rightarrow F$ by $\bar{\alpha}{\mid}_{S_k^{n-1}}=\tilde{\alpha}{\mid}_{S_k^{n-1}}(-,1)$ which is continuous by Lemma \ref{joinmap}. Define $\partial([\alpha])=\bar{\alpha}$, then similar to the proof of \cite[Theorem 15]{luke}, one can verifies that the sequence is exact.
\end{proof}

The following corollary is an immediate consequence of the above theorem.
\begin{corollary}
Let  $p:\tilde{X}\rightarrow X$ be a covering, $\tilde{X}$ be first countable at $\tilde{x}_0$ and $n\geq 2$. Then
${\mathcal{H}}_{n}(\tilde{X},\tilde{x}_0)\cong {\mathcal{H}}_{n}(X,x_0)$, where $\tilde{x}_0\in p^{-1}(x_0)$. Moreover,
$p_{\ast}:{\mathcal{H}}_{1}(\tilde{X},\tilde{x}_0)\rightarrow{\mathcal{H}}_{1}(X,x_0)$ is a monomorphism.
\end{corollary}

\begin{theorem}
Let $p:E\rightarrow B$ be a locally trivial bundle, $y_0\in B$ and $x_0\in p^{-1}(y_0)=F$. If ${\pi}_n(E,x_0)$ and
${\pi}_n(B,y_0)$ are trivial, then $${\mathcal{H}}_n(E,x_0)\cong {\mathcal{H}}_n(B,y_0)\times {\mathcal{H}}_n(F,x_0).$$
\end{theorem}
\begin{proof}
Let $[f]\in{\mathcal{H}}_n(E,x_0)$ and ${\sigma}_{0}:U\rightarrow V\times F$ be the local trivialization at $x_0$. There
exists $K\in\Bbb{N}$ such that if $k\geq K$, then $Im(f{\mid}_{S_k^n})\subseteq U$. Consider $p_1:B\times F\rightarrow B$,
$p_2:B\times F\rightarrow F$  as the projection maps on first and second components, respectively, $i:V\rightarrow B$ as
the inclusion map and $R_K:{\Bbb{H}}^n\rightarrow {\Bbb{H}}^n$ as the natural contraction on ${\tilde{\bigvee}}_{k\geq K}
S_k^n$.
 We define $\psi:{\mathcal{H}}_n(E,x_0)\rightarrow {\mathcal{H}}_n(B,y_0)\times {\mathcal{H}}_n(F,x_0)$ by $\psi([f])=
 ([p_1\circ i\circ{\sigma}_{0}\circ f\circ R_K],[p_2\circ{\sigma}_{0}\circ f\circ R_K])$. We first show that $\psi $ is
 independent of the choice of $K$. To prove this, let $K_1,K_2$ be two natural numbers such that $Im(f
 {\mid}_{{\tilde{\bigvee}}_{k\geq K_1}S_k^n})\subseteq U$ and
 $Im(f{\mid}_{{\tilde{\bigvee}}_{k\geq K_2}S_k^n})\subseteq U$. We prove  that $[f\circ R_{K_1}]=[f\circ R_{K_2}]$.
 Without loss of generality, let $K_1< K_2$. Since ${\pi}_n(E,x_0)$ is trivial, $f{\mid}_{S_k^n}$ is null-homotopic relative to $\{\theta\}$,
 for each $k\in \Bbb{N}$, in particular $K_1\leq k<K_2$. By Lemma \ref{joinmap} $(iii)$ we
  have $f\circ R_{K_1}\cong f\circ R_{K_2}$ $rel\ \{\theta\}$, so
   $p_1\circ i\circ {\sigma}_{0}\circ f\circ R_{K_1}\cong p_1\circ i\circ{\sigma}_{0}\circ f\circ R_{K_2}$ and
   $p_2\circ {\sigma}_{0}\circ f\circ R_{K_1}\cong p_2\circ {\sigma}_{0}\circ f\circ R_{K_2}$ $rel\ \{\theta\}$.
   Thus $\psi([f])$ is unique and hence $\psi$ is well defined.

We show that $\psi$ is a homomorphism. Let $[f],[g]\in{\mathcal{H}}_n(E,x_0)$ and for $K$ we have
$Im(f{\mid}_{{\tilde{\bigvee}}_{k\geq K}S_k^n})\bigcup Im(g{\mid}_{{\tilde{\bigvee}}_{k\geq K}S_k^n})\subseteq U$. Then
$$\psi([f\ast g])=([p_1\circ i\circ{\sigma}_{0}\circ(f\ast g)\circ R_K],[p_2\circ{\sigma}_{0}\circ(f\ast g)\circ R_K])$$
$$=([p_1\circ i\circ {\sigma}_{0}\circ(f\circ R_K\ast g\circ R_K)],[p_2\circ{\sigma}_{0}\circ(f\circ R_K\ast g\circ R_K)])$$
$$=([p_1\circ i\circ {\sigma}_{0}\circ f\circ R_K]\ast[p_1\circ i\circ{\sigma}_{0}\circ g\circ R_K],
[p_2\circ{\sigma}_{0}\circ f\circ R_K]\ast[p_2\circ {\sigma}_{0}\circ g\circ R_K])$$ $$=\psi([f])\ast\psi([g]).$$

We show that $\psi$ is an epimorphism. Let $([h_1],[h_2])\in{\mathcal{H}}_n(B,y_0)\times {\mathcal{H}}_n(F,x_0)$, there
exists $K\in\Bbb{N}$ such that $Im(h_1{\mid}_{{\tilde{\bigvee}}_{k\geq K}S_k^n})\subseteq V$. Let
$h:\Bbb{H}^n\rightarrow V\times F$ be the unique map  with $p_1\circ h=h_1\circ R_K$ and $p_2\circ h=h_2$. Since
${\pi}_n(B,y_0)$ is trivial, $h_1\circ R_K\cong h_1$. We have $Im({\sigma}_{0}^{-1}\circ h)\subseteq U$, so
$\psi([{\sigma}_{0}^{-1}\circ h])=([p_1\circ i\circ {\sigma}_{0}\circ({\sigma}_{0}^{-1}\circ h)],[p_2\circ {\sigma}_{0}
\circ({\sigma}_{0}^{-1}\circ h)])=([h_1],[h_2])$.

Finally, since $({p_1}_{\ast},{p_2}_{\ast})$ and ${{\sigma}_{0}}_{\ast}$ are isomorphisms, we can easily see that
$\psi$ is injective and hence it is an isomorphism.
\end{proof}

\section{Infinite Dimensional Hawaiian Groups}
\begin{theorem}\label{lsc}
Let $(X,x_0)$ be a first countable, locally strongly contractible pointed space, then the following isomorphism holds
$${\mathcal{H}}_{\infty}(X,x_0)\cong{\prod}_{n\in\Bbb{N}}^w{\prod}_{k\in\Bbb{N}}^w{\pi}_n(X,x_0).$$
\end{theorem}
\begin{proof}
We define a homomorphism $\phi:{\mathcal{H}}_{\infty}(X,x_0)\rightarrow{\prod}_{n\in\Bbb{N}}{\mathcal{H}}_n(X,x_0)$ by
$\phi([\alpha])=([\alpha{\mid}_{{\Bbb{H}}^1}],[\alpha{\mid}_{{\Bbb{H}}^2}],...)$. Similar to the proof of \cite[Theorem 1]{karh}
we can prove that $Im(\phi)={\prod}_{n\in\Bbb{N}}^w{\mathcal{H}}_n(X,x_0)$. Also, we can see that $\phi$
is injective and hence it is an isomorphism onto ${\prod}_{n\in\Bbb{N}}^w{\mathcal{H}}_n(X,x_0)$. Using \cite[Theorem 1]{karh},
we obtain an isomorphism  from ${\mathcal{H}}_{\infty}(X,x_0)$ onto
${\prod}_{n\in\Bbb{N}}^w{\prod}_{k\in\Bbb{N}}^w{\pi}_n(X,x_0)$.
\end{proof}
	

\begin{definition}
 We call a pointed space $(X,x_0)$ to be locally infinite-connected if for each open set $U$ contains $x_0$, there exists
 a neighborhood $V$ of $x_0$ such that the inclusion map $V\hookrightarrow U$, induces trivial homomorphism on $n$th
 homotopy group, for all $n\in\Bbb{N}$ .
\end{definition}


\begin{theorem}\label{infconn}
Let $(X,x_0)$ be a first countable locally infinite-connected pointed space, then
$${\mathcal{H}}_{\infty}(X,x_0)\cong{\prod}_{n\in\Bbb{N}}^w{\prod}_{k\in\Bbb{N}}^w{\pi}_n(X,x_0).$$
\end{theorem}
\begin{proof}
Consider the homomorphism $\phi:{\mathcal{H}}_{\infty}(X,x_0)\rightarrow{\prod}_{n\in\Bbb{N}}{\mathcal{H}}_n(X,x_0)$
defined in the proof of Theorem \ref{lsc}. Similar to the proof of \cite[Theorem 1]{karh} we can show that
$\phi:{\mathcal{H}}_{\infty}(X,x_0)\rightarrow{\prod}_{n\in\Bbb{N}}^w{\mathcal{H}}_n(X,x_0)$ is an isomorphism.
Using \cite[Theorem 1]{karh}, we obtain an isomorphism from ${\mathcal{H}}_{\infty}(X,x_0)$ onto
${\prod}_{n\in\Bbb{N}}^w{\prod}_{k\in\Bbb{N}}^w{\pi}_n(X,x_0)$.
\end{proof}


\begin{theorem}\label{pr00}
For  every family of pointed spaces $\{(X_i,x_i)\}_{i\in I}$ the following isomorphism holds
$${\mathcal{H}}_\infty({\prod}_{i\in I}X_i,x_\ast)\cong {\prod}_{i\in I} {\mathcal{H}}_\infty(X_i,x_i),$$
where $x_\ast =\{x_i\}_{i\in I}\in {\prod}_{i\in I}X_i$ .
\end{theorem}
\begin{proof}
See the proof of Theorem \ref{pr}.
\end{proof}


\begin{lemma}\label{weakdirect}
If $(X,x_0)$ is a pointed space, then ${\prod}_{i\in \Bbb{N}}^w{\mathcal{H}}_n(X,x_0)$ can be embedded in
${\mathcal{H}}_{\infty}(X,x_0)$.
\end{lemma}
\begin{proof}
Similar to the proof of Lemma \ref{weakdir} we define
$\rho:{\prod}_{n\in \Bbb{N}}^w{\mathcal{H}}_n(X,x_0)\rightarrow {\mathcal{H}}_{\infty}(X,x_0)$ by the rule
$$\rho([f_1],[f_2],...,[f_m],e,e,...)=[f],$$ where $f{\mid}_{{\Bbb{H}}^n}=f_n$, for $n\leq m$, and
 $f{\mid}_{{\Bbb{H}}^n}=c_{x_0}$, for $n>m$.
By the topology of $\Bbb{H}^{\infty}$  $f$ is continuous and $\rho$ is well-defined.
We can easily see that $\rho$ is a monomorphism.
\end{proof}


The following corollary is a consequence of Lemmas \ref{weakdirect} and \ref{weakdir}.
\begin{corollary}
If $(X,x_0)$ is a pointed space, then ${\prod}_{n\in \Bbb{N}}^w{\prod}_{k\in \Bbb{N}}^w{\pi}_n(X,x_0)$ can be
embedded in ${\mathcal{H}}_{\infty}(X,x_0)$.
\end{corollary}


\begin{theorem}\label{3.7}
Let $(X,x_0)$ be a pointed space and $C(X)$ be the cone over $X$. Suppose that  $\hat{x}_t= (x_0,t)$ is a point of $C(X)$
except the vertex ($t\neq 1$). Then $${\mathcal{H}}_{\infty}(C(X),\hat{x}_t)\cong\frac{{\mathcal{H}}_{\infty}(X,x_0)}
{{\prod}_{n\in \Bbb{N}}^w{\prod}_{k\in \Bbb{N}}^w{\pi}_n(X,x_0)}.$$
\end{theorem}
\begin{proof}
Let $\mu={\mathcal{H}}_{\infty}(i_t):{\mathcal{H}}_{\infty}(X,x_0)\rightarrow {\mathcal{H}}_{\infty}(C(X),x_t)$
be defined by  $\mu([f])=[i_t\circ f]$, where $i_t$ is defined in the proof of Theorem \ref{cone}.
Similar to the proof of Theorem \ref{cone}, we can prove that $\mu$ is an epimorphism with
kernel ${\prod}_{n\in \Bbb{N}}^w{\prod}_{k\in \Bbb{N}}^w{\pi}_n(X,x_0)$.
\end{proof}


\begin{example}
Let $S={\widetilde{\bigvee}}_{n\in{\Bbb{N}}} S^n$ be the weak join of the family $\{S^n\mid n\in \Bbb{N}\}$ with the common point $\eta$.
Then ${\mathcal{H}}_{\infty}(CS,\hat{\eta})$ is not trivial, where $\hat{\eta}=(\eta, 0)$.
\end{example}
\begin{proof}
Consider the map
$\nu:{\Bbb{H}}^{\infty}\rightarrow S$ which is defined as follows. If $r\in S_n^n$, then $\nu(r)=r$ and $\nu(r)=\hat{\eta}$ otherwise.
If the equivalence class of this map belongs to ${\prod}_{n\in \Bbb{N}}^w{\prod}_{k\in \Bbb{N}}^w{\pi}_n(S,\eta)$, then
the identity map $id_{S^n}$ is null-homotopic for some $n\in\Bbb{N}$ which is a contradiction. Thus by Theorem \ref{3.7} ${\mathcal{H}}_{\infty}(CS,\hat{\eta})$ is not
trivial.
\end{proof}


\begin{theorem}\label{main2}
Let ${\{(X_i,x_i)\}}_{i\in{\Bbb{N}}}$ be a family of locally strongly contractible, first countable pointed spaces.
If $X={\widetilde{\bigvee}}_{i\in{\Bbb{N}}}(X_i,x_i)$, then ${\mathcal{H}}_{\infty}(X,x_\ast) $ can be embedded in
${\prod}_{n\in\Bbb{N}}{\prod}_{k\in\Bbb{N}}{\pi}_n(X,x_{\ast})$.
\end{theorem}
\begin{proof}
Let  $\phi:{\mathcal{H}}_{\infty}(X,x_\ast)\rightarrow {\prod}_{n\in\Bbb{N}}{\mathcal{H}}_n(X,x_\ast)$ be the
homomorphism defined in the proof of Theorem \ref{main}. To show the injectivity of the $\phi$, let $[f]$ be an element
of \ ${\mathcal{H}}_{\infty}(X,x_\ast)$ so that $\phi([f])=(e,e,...)$.
Therefore $f{\mid}_{{\Bbb{H}}^n}\simeq c \ \ rel\{\theta\}$ with a homotopy mapping
$F_n:{{\Bbb{H}}^n}\times I\rightarrow X$. Let $\{U_m\}$ be the countable basis at $x_{\ast}$ as in the proof of Theorem
\ref{main}. For any $m\in \Bbb{N}$, there is $K_m$ such that if $n\geq K_m$, then
$Im(f{\mid}_{{\Bbb{H}}^n})\subseteq U_{m}$.
The proof of Theorem \ref{main} implies that $f{\mid}_{{\Bbb{H}}^n}$ is  homotopic  to a map
$f_n$ relative to $\{\theta\}$ in $U_m$, where $f_n:=R_m\circ f{\mid}_{{\Bbb{H}}^n}$  and so
$Im(f_n)\subseteq \widetilde{{\bigvee}}_{i\geq m}X_i$. Now $R_m\circ F_n$ makes $f_n$ to be  null-homotopic
$(rel \{\theta\})$.
Joining these homotopies construct a pointed homotopy  between $\tilde{f}$ ( i.e the join of the $f_n$'s) and the constant map.
Since $f$ is  homotopic to $\tilde{f}$ $(rel \{\theta\})$, $f$ is  null-homotopic $(rel \{\theta\})$,
as required. Applying Theorem
\ref{main}, ${\mathcal{H}}_{\infty}(X,x_\ast)$ is isomorphic to a subgroup of
${\prod}_{n\in\Bbb{N}}{\prod}_{k\in\Bbb{N}}{\pi}_n(X,x_{\ast})$.
\end{proof}

\subsection*{Acknowledgements}
The authors would like to thank the referee for the valuable comments that help improve the manuscript.\\
This research was supported by a grant from Ferdowsi University of Mashhad (No. MP90237HMP).












\end{document}